\documentclass{amsart}

\usepackage[T1]{fontenc}
\usepackage[utf8]{inputenc}
\usepackage{lmodern}
\usepackage{microtype}
\usepackage{mathtools}
\usepackage{amssymb}
\usepackage{enumitem}
\usepackage{placeins}
\usepackage{xcolor}
\usepackage{tikz}
\usetikzlibrary{arrows.meta,positioning}
\usepackage[bookmarks,colorlinks,breaklinks]{hyperref}
\usepackage[nameinlink,capitalise,noabbrev]{cleveref}

\usepackage{graphics,colortbl}

\setlist{nosep}

\newcommand{\N}{\mathcal N}

\newcommand{\FSIN}{\operatorname{FSIN}}
\newcommand{\SIN}{\operatorname{SIN}}

\newcommand{\e}{e}

\tikzset{
  qnode/.style={draw=blue!55!black,rounded corners=2pt,
    minimum width=17mm,minimum height=8mm,align=center,
    inner sep=2pt,font=\small},
  openq/.style={qnode,fill=gray!8},
  provedq/.style={qnode,fill=green!15,very thick},
  frontierq/.style={qnode,fill=yellow!20,very thick},
  falseq/.style={qnode,fill=red!10,double,double distance=0.7pt},
  implication/.style={-{Latex[length=2mm]},semithick,
    draw=blue!55!black},
  equivalence/.style={Latex-Latex,semithick,draw=blue!55!black}
}

\theoremstyle{plain}
\newtheorem{theorem}{Theorem}[section]

\newtheorem{lemma}[theorem]{Lemma}
\newtheorem{corollary}[theorem]{Corollary}
\newtheorem{criterion}[theorem]{Criterion}

\theoremstyle{definition}
\newtheorem{definition}[theorem]{Definition}
\newtheorem{question}{Question}

\theoremstyle{remark}

\crefname{criterion}{criterion}{criteria}
\crefname{question}{Question}{Questions}

\title[Countable uniformly discrete sets in FSIN groups]
  {Countable uniformly discrete sets in functionally balanced groups}

\author{Dekui Peng}
\address{Institute of Mathematics, Nanjing Normal University,
  Nanjing 210024, China}
\email{pengdk10@lzu.edu.cn}

\author{Li-Hong Xie}
\address{School of Mathematics and Computational Science, Wuyi University,
  Jiangmen, Guangdong 529000, P.R. China}
\email{yunli198282@126.com}

\author{Jiang Yang}
\address{School of Mathematical Sciences, Guangxi Minzu University,
  Nanning 530006, P.R. China}
\email{yangjjiangdy@126.com}
\thanks{Jiang Yang is the corresponding author.}

\date{August 13, 2026}

\subjclass[2020]{Primary 22A05; Secondary 54E15, 54H11, 43A10}
\keywords{Functionally balanced group, SIN group, $\omega$-narrow group,
  uniformly discrete set, thin set, precompact set,
  left and right uniformities}

\begin{document}

\begin{abstract}
We prove that every countable left uniformly discrete subset of a Hausdorff
functionally balanced topological group is right thin. As applications, we
answer two questions of Bouziad and Troallic: every left precompact subset of
a Hausdorff functionally balanced group is right precompact, and every
Hausdorff $\omega$-narrow functionally balanced group is a SIN group. 
\end{abstract}

\maketitle

\section{Introduction}

Let $G$ be a Hausdorff topological group. The left and right uniformities of
$G$ are generated, respectively, by the entourages
$\{(x,y)\in G\times G:x^{-1}y\in V\}$ and $\{(x,y)\in G\times G:xy^{-1}\in V\}$, where
$V$ ranges over the identity neighbourhoods of $G$. These two uniformities
need not coincide. A group for which they do coincide is called
\emph{balanced}, or a \emph{SIN group}. Equivalently, every identity
neighbourhood of a SIN group contains an identity neighbourhood invariant
under all inner automorphisms.

The comparison of these two uniformities through uniformly continuous
real-valued functions goes back to Comfort and Ross. In 1966 they asserted
that if every real-valued left uniformly continuous function on a
topological group is right uniformly continuous, then the left and right
uniformities coincide \cite{ComfortRoss1966}. Following Protasov, a
topological group is called \emph{functionally balanced}, or \emph{FSIN},
if every bounded real-valued left uniformly continuous function is right
uniformly continuous. Thus every SIN group is FSIN. The converse is the
bounded form of the classical Itzkowitz problem, and was singled out by
Bouziad and Troallic as their main problem.

\begin{question}\cite[Question~2]{BouziadTroallic2007}
\label{q:itzkowitz}
Is $[\FSIN]=[\SIN]$?
\end{question}

A substantial part of the literature has been devoted to
Question~\ref{q:itzkowitz} under additional hypotheses. The locally compact
case was obtained independently by Itzkowitz, Milnes, and Protasov;
Protasov proved the result more generally for almost metrizable groups,
while Megrelishvili, Nickolas, and Pestov treated locally connected groups.
Bouziad and Troallic later developed the approach through thin subsets and
almost balanced groups; in particular, locally precompact FSIN groups are
SIN. See
\cite{Itzkowitz1991,Milnes1990,Protasov1991,
MegrelishviliNickolasPestov1997,BouziadTroallic2004,
BouziadTroallic2006,BouziadTroallic2007} and the references therein.
Nevertheless, Question~\ref{q:itzkowitz} remains open in full generality.

Among the concrete problems proposed by Bouziad and Troallic, two are
especially closely related to our main result. The first asks whether
functional balance is already sufficient to eliminate the asymmetry between
left and right precompactness. Recall that a subset $P\subseteq G$ is left
precompact if, for every $V\in\N_G(\e)$, there is a finite $F\subseteq G$
with $P\subseteq FV$; right precompactness is defined analogously using
covers $P\subseteq VF$.

\begin{question}\cite[Question~8]{BouziadTroallic2007}
\label{q:precompact}
Let $G$ be a functionally balanced group. Is every left precompact subset of
$G$ right precompact?
\end{question}

A positive answer to Question~\ref{q:itzkowitz} would immediately imply a
positive answer to Question~\ref{q:precompact}, because the two notions of
precompactness agree in a SIN group. Question~\ref{q:precompact} is,
however, a statement about individual subsets and is therefore potentially
weaker than the full equality of the left and right uniformities.

The second problem restricts Question~\ref{q:itzkowitz} to
$\aleph_0$-bounded groups.

\begin{question}\cite[Question~12]{BouziadTroallic2007}
\label{q:omega}
Let $G$ be a member of $[\FSIN]$ which is $\aleph_0$-bounded. Is $G$
balanced?
\end{question}

We use the now standard term \emph{$\omega$-narrow} for
$\aleph_0$-bounded. Thus $G$ is $\omega$-narrow if for every
$V\in\N_G(\e)$ there is a countable set $C\subseteq G$ such that $G=CV$.
By Guran's theorem, these are precisely the topological subgroups of
products of second-countable topological groups \cite{Guran1981}.
Consequently, this class contains all separable, Lindel\"of, and
$\sigma$-compact topological groups. Thus Question~\ref{q:omega} is a broad
and product-stable special case of Question~\ref{q:itzkowitz}.

The connection with uniformly discrete subsets is provided by a criterion
of Hern\'andez, rooted in earlier work of Itzkowitz: a topological group is
SIN if and only if every left uniformly discrete subset is right thin
\cite{Hernandez2000}. Here $A\subseteq G$ is left uniformly discrete if the
sets $aO$, $a\in A$, are pairwise disjoint for some identity neighbourhood
$O$, and it is right thin if for every $V\in\N_G(\e)$ there is
$U\in\N_G(\e)$ such that $a^{-1}Ua\subseteq V$ for all $a\in A$. Our main
result establishes this thinness for every countable left uniformly
discrete subset of an FSIN group.

\begin{theorem}\label{thm:countable-thin}
Let $G$ be a Hausdorff FSIN group. Then every countable left uniformly
 discrete subset $A\subseteq G$ is right thin. Explicitly, for every
$V\in\N_G(\e)$ there is $U\in\N_G(\e)$ such that
\begin{equation}\label{eq:thin-intro}
                         a^{-1}Ua\subseteq V
                         \qquad(a\in A).
\end{equation}
\end{theorem}

Theorem~\ref{thm:countable-thin} yields affirmative answers to both
Questions~\ref{q:precompact} and~\ref{q:omega}. More precisely, we prove
that in every Hausdorff FSIN group, every left precompact subset is right
precompact; hence left and right precompact subsets coincide. We also prove
that every Hausdorff $\omega$-narrow FSIN group is SIN. In particular, the
latter applies to FSIN groups that are separable, Lindel\"of,
$\sigma$-compact, or countable. These two consequences are stated and
proved in Section~\ref{sec:applications}. The proof of
Theorem~\ref{thm:countable-thin} is postponed to the final section.

\subsection{Uniformity criteria and terminology}

All topological groups in the paper are Hausdorff.  We denote the identity
of a group by $\e$ and its identity-neighbourhood filter by $\N_G(\e)$.

\begin{definition}\label{def:lud-thin}
A subset $A\subseteq G$ is \emph{left uniformly discrete} if there is
$O\in\N_G(\e)$ such that the sets $aO$, $a\in A$, are pairwise disjoint.
It is \emph{right uniformly discrete} if there is $O\in\N_G(\e)$ such
that the sets $Oa$, $a\in A$, are pairwise disjoint.  It is \emph{right
thin} if for every $V\in\N_G(\e)$ there is $U\in\N_G(\e)$ such that
$a^{-1}Ua\subseteq V$ for all $a\in A$.
\end{definition}

The last condition is equivalent to saying that
$\bigcap_{a\in A}aVa^{-1}$ is an identity neighbourhood for every
$V\in\N_G(\e)$.  We use the explicit conjugation formula because the words
\emph{left thin} and \emph{right thin} are interchanged in parts of the
literature.

A subset $P\subseteq G$ is \emph{left precompact} if for every
$V\in\N_G(\e)$ there is a finite $F\subseteq G$ with $P\subseteq FV$;
it is \emph{right precompact} if for every $V\in\N_G(\e)$ there is a
finite $F\subseteq G$ with $P\subseteq VF$.  Inversion interchanges these
two notions.  We call a subset \emph{precompact} if it is both left and
right precompact.  The two one-sided notions need not agree in general,
but they agree for symmetric subsets and for all subsets of a SIN group.
A group is \emph{locally precompact} if it has a precompact identity
neighbourhood.

We shall use the following three known results.

\begin{criterion}[Protasov--Saryev]\label{crit:ps}
A topological group $G$ is FSIN if and only if, for every $C\subseteq G$
and every $T\in\N_G(\e)$, there is $U\in\N_G(\e)$ such that
$UC\subseteq CT$.
\end{criterion}

The criterion is due to Protasov and Saryev
\cite{ProtasovSaryev1988,Protasov1991}; the above orientation is stated
explicitly in \cite[Theorem~4.1]{MbomboPestov2012}.  Some sources use the
equivalent form $CU\subseteq TC$.  Applying that form to $C^{-1}$ and
$T^{-1}$ and then taking inverses gives the version above.

\begin{criterion}[Hern\'andez--Itzkowitz]\label{crit:thin}
A topological group $G$ is SIN if and only if every left uniformly
discrete subset of $G$ is right thin.
\end{criterion}

This is the equivalence recorded in \cite[Lemma~2]{Hernandez2000}; that
paper also notes the earlier role of Itzkowitz.  The criterion carries no
countability or metrizability hypothesis.

\begin{theorem}[Bouziad--Troallic]\label{thm:local-precompact}
Every locally precompact FSIN group is SIN.
\end{theorem}

This is \cite[Corollary~3.7]{BouziadTroallic2004}.  In particular, it
contains the locally compact case.

\section{Applications of Theorem \ref{thm:countable-thin}}\label{sec:applications}

We first answer Question~\ref{q:precompact}.

\begin{theorem}\label{thm:q8}
Let $G$ be a Hausdorff FSIN group. Every left precompact subset of $G$ is
right precompact. Consequently, left and right precompact subsets of $G$
coincide.
\end{theorem}

\begin{proof}
Let $P\subseteq G$ be left precompact and suppose that $P$ is not right
precompact.  Then there is $W\in\N_G(\e)$ such that
$P\not\subseteq WF$ for every finite $F\subseteq G$.  Choose a symmetric
$V\in\N_G(\e)$ with $V^2\subseteq W$.  Recursively choose
$d_n\in P\setminus\bigcup_{i<n}V^2d_i$.  This is possible because
$\bigcup_{i<n}V^2d_i\subseteq W\{d_i:i<n\}$.  If $i<j$ and
$Vd_i\cap Vd_j\neq\varnothing$, then $d_j\in V^2d_i$, a contradiction.
Thus $D=\{d_n:n\in\omega\}$ is right uniformly discrete, and hence
$D^{-1}$ is countable and left uniformly discrete.

By Theorem~\ref{thm:countable-thin}, $D^{-1}$ is right thin.  Applying
thinness with the target neighbourhood $V$, and shrinking if necessary,
choose a symmetric $U\in\N_G(\e)$ such that $dUd^{-1}\subseteq V$ for
every $d\in D$.  Hence $dU\subseteq Vd$ for all $d\in D$, so the sets
$dU$, $d\in D$, are pairwise disjoint.

Choose a symmetric $R\in\N_G(\e)$ with $R\subseteq U$.  Since $P$ is left
precompact, there is a finite $F\subseteq G$ such that $P\subseteq FR$.
Two distinct points $d_i,d_j\in D$ must lie in the same translate $fR$.
Writing $d_k=fr_k$ with $r_k\in R$ for $k=i,j$, we obtain
$f=d_kr_k^{-1}\in d_kR\subseteq d_kU$ for both $k=i$ and $k=j$, contrary
to the pairwise disjointness of the sets $dU$.  Therefore $P$ is right
precompact.

Conversely, if $P$ is right precompact, then $P^{-1}$ is left precompact,
so by what we have just proved $P^{-1}$ is right precompact.  Taking
inverses shows that $P$ is left precompact.  Thus the two notions coincide.
\end{proof}

For Question~\ref{q:omega} we need only the following elementary observation.

\begin{lemma}\label{lem:omega-lud}
Every left uniformly discrete subset of an $\omega$-narrow topological
group is countable.
\end{lemma}

\begin{proof}
Let $A\subseteq G$ be left uniformly discrete, and choose
$R\in\N_G(\e)$ such that the sets $aR$, $a\in A$, are pairwise disjoint.
Choose a symmetric $S\in\N_G(\e)$ with $S^{-1}S\subseteq RR^{-1}$.  By
$\omega$-narrowness, $G=FS$ for some countable $F\subseteq G$.  Each
translate $fS$ meets $A$ in at most one point: if distinct
$a,b\in A\cap fS$, then $a^{-1}b\in S^{-1}S\subseteq RR^{-1}$, which is
equivalent to $aR\cap bR\neq\varnothing$.  Hence $A$ is countable.
\end{proof}

\begin{theorem}\label{thm:q12}
Every Hausdorff $\omega$-narrow FSIN topological group is SIN.
\end{theorem}

\begin{proof}
Let $G$ be Hausdorff, $\omega$-narrow, and FSIN.  By
Lemma~\ref{lem:omega-lud}, every left uniformly discrete subset of $G$ is
countable, and by Theorem~\ref{thm:countable-thin} every such subset is
right thin.  Criterion~\ref{crit:thin} therefore implies that $G$ is SIN.
\end{proof}

The product description of $\omega$-narrow groups and several familiar
covering properties give immediate special cases.

\begin{corollary}\label{cor:scope}
Each of the following Hausdorff FSIN groups is SIN:
\begin{enumerate}[label=\textup{(\roman*)}]
\item every topological subgroup of a product of second-countable groups;
\item every separable topological group;
\item every Lindel\"of topological group;
\item every $\sigma$-compact topological group;
\item every countable topological group.
\end{enumerate}
\end{corollary}

\begin{proof}
Item~\textup{(i)} is Guran's characterization \cite{Guran1981}.  If $G$
has a countable dense subset $D$ and $V\in\N_G(\e)$, choose a symmetric
open identity neighbourhood $W\subseteq V$.  Every translate $gW$ meets
$D$, and therefore $G=DW\subseteq DV$; hence every separable group is
$\omega$-narrow.  If $G$ is Lindel\"of, choose an open identity
neighbourhood $W\subseteq V$.  The cover $\{gW:g\in G\}$ has a countable
subcover, so $G=FW\subseteq FV$ for some countable $F\subseteq G$.  A
$\sigma$-compact space is Lindel\"of, and a countable group is trivially
$\omega$-narrow.  The conclusion follows from Theorem~\ref{thm:q12}.
\end{proof}

\begin{corollary}\label{cor:counterexample-boundary}
If a Hausdorff FSIN group is not SIN, then it is not $\omega$-narrow and
it contains an uncountable left uniformly discrete subset which is not
right thin.
\end{corollary}

\begin{proof}
The first assertion is the contrapositive of Theorem~\ref{thm:q12}.  By
Criterion~\ref{crit:thin}, a non-SIN group has a left uniformly discrete
subset which is not right thin.  Such a subset cannot be countable by
Theorem~\ref{thm:countable-thin}.
\end{proof}

\section{Proof of Theorem \ref{thm:countable-thin}}

We first isolate the elementary consequence of non-precompactness used in
the proof.  The direction of the translates in this lemma is important.

\begin{lemma}\label{lem:separated}
Let $N=N^{-1}$ be an identity neighbourhood in a topological group. If
$N$ is not precompact, then there exist a countably infinite set
$D\subseteq N$ and a symmetric identity neighbourhood $Q$ such that the
sets $dQ$, $d\in D$, are pairwise disjoint.
\end{lemma}

\begin{proof}
Since $N=N^{-1}$, finite left-cover and finite right-cover precompactness
are equivalent for $N$.  Hence there is a symmetric identity neighbourhood
$P$ such that $N\not\subseteq FP$ for every finite $F\subseteq G$.  Choose
a symmetric identity neighbourhood $Q$ with $Q^2\subseteq P$.  Recursively
select $d_n\in N\setminus\bigcup_{i<n}d_iP$.  If
$d_iQ\cap d_jQ\neq\varnothing$ for $i<j$, then
$d_j\in d_iQQ^{-1}\subseteq d_iP$, contrary to the construction.  Thus
$D=\{d_n:n\in\omega\}$ has the required properties.
\end{proof}

\begin{proof}[Proof of Theorem \ref{thm:countable-thin}]
The assertion is trivial if $A$ is finite.  Assume that $A$ is countably
infinite.  If $G$ is locally precompact, then
Theorem~\ref{thm:local-precompact} implies that $G$ is SIN, and hence every
subset of $G$ is right thin.  We may therefore suppose that $G$ is not
locally precompact.  In particular, no identity neighbourhood is
precompact.

Enumerate $A$ without repetitions as $A=\{a_n:n\in\omega\}$ and fix
$V\in\N_G(\e)$.  By left uniform discreteness, choose a symmetric
$O\in\N_G(\e)$ such that $O\subseteq V$ and the sets $a_nO$ are pairwise
disjoint.  Choose a symmetric $W\in\N_G(\e)$ with $W^3\subseteq O$.

The neighbourhood $W$ is not precompact.  By Lemma~\ref{lem:separated},
there are a countably infinite set $D\subseteq W$ and a symmetric identity
neighbourhood $Q$ such that the sets $dQ$, $d\in D$, are pairwise
disjoint.  Shrinking $Q$ if necessary, assume also that $Q\subseteq W$.
Choose pairwise disjoint finite sets $D_n\subseteq D$ with $|D_n|=n+2$,
and put $C_n=a_nD_n$ and $C=\bigcup_{n\in\omega}C_n$.  Since
$D_n\subseteq W\subseteq O$, the sets $C_n$ are pairwise disjoint.

We claim that the sets $cQ$, $c\in C$, are pairwise disjoint.  Let
$c_1=a_nd_1$ and $c_2=a_md_2$ be distinct, with $d_1\in D_n$ and
$d_2\in D_m$.  If $n\neq m$, then $c_1Q\subseteq a_nW^2\subseteq a_nO$
and $c_2Q\subseteq a_mW^2\subseteq a_mO$, so they are disjoint.  If
$n=m$, then $d_1\neq d_2$, and the disjointness of $d_1Q$ and $d_2Q$
gives $c_1Q\cap c_2Q=\varnothing$.

By Criterion~\ref{crit:ps}, after shrinking if necessary, there is a
symmetric $U\in\N_G(\e)$ such that $UC\subseteq CQ$.  Fix $u\in U$.
For each $x\in C$, the pairwise disjointness of the sets $cQ$ gives a
unique point $\phi_u(x)\in C$ such that $ux\in\phi_u(x)Q$.

The map $\phi_u:C\to C$ is a permutation and
$\phi_{u^{-1}}=\phi_u^{-1}$.  Indeed, if $ux=yq$ with
$y=\phi_u(x)$ and $q\in Q$, then $u^{-1}y=xq^{-1}\in xQ$.  Since
$u^{-1}\in U$ and $Q$ is symmetric, uniqueness gives
$\phi_{u^{-1}}(y)=x$.  Hence
$\phi_{u^{-1}}\circ\phi_u=\operatorname{id}_C$, and the same argument
with $u^{-1}$ in place of $u$ gives the reverse composition.

We next show that $\phi_u(C_n)=C_n$ for every $n\in\omega$.  Fix $n$ and
take $x_i=a_nd_i\in C_n$, where $d_i\in D_n$, for $i=1,2$.  Write
$\phi_u(x_i)=a_{m_i}d_i'$ with $d_i'\in D_{m_i}$.  From
$ua_nd_i\in a_{m_i}d_i'Q$ we obtain
$ua_n\in a_{m_i}d_i'Qd_i^{-1}\subseteq a_{m_i}W^3\subseteq a_{m_i}O$.
Thus $ua_n\in a_{m_1}O\cap a_{m_2}O$, and the disjointness of the sets
$a_kO$ gives $m_1=m_2$.  Therefore, for each $n$, there is some $m$ such
that $\phi_u(C_n)\subseteq C_m$.  Since $\phi_u$ is injective and
$|C_k|=k+2$, we have $n\leq m$.

Take $x\in C_n$ and put $y=\phi_u(x)\in C_m$.  Applying the preceding
argument to $u^{-1}$ and $C_m$, there is some $n'$ such that
$\phi_{u^{-1}}(C_m)\subseteq C_{n'}$, so $m\leq n'$.  But
$x=\phi_{u^{-1}}(y)\in C_{n'}\cap C_n$, and the sets $C_k$ are pairwise
disjoint.  Hence $n'=n$, so $m\leq n$.  Consequently $m=n$, and
$\phi_u(C_n)=C_n$ because $C_n$ is finite.

It follows that for every $u\in U$, $n\in\omega$, and $d\in D_n$, we
have $ua_nd\in a_nD_nQ$.  Hence
$a_n^{-1}ua_n\in D_nQd^{-1}\subseteq DQD^{-1}\subseteq W^3\subseteq
O\subseteq V$.  Therefore $a_n^{-1}Ua_n\subseteq V$ for every
$n\in\omega$.  Equivalently, $U\subseteq\bigcap_{a\in A}aVa^{-1}$, and
$A$ is right thin.
\end{proof}

\end{document}